\documentclass[11pt]{amsart}
\usepackage{amsmath}
\usepackage{mathrsfs}
\usepackage{cases}
\usepackage{pifont}%添加带有圆圈的数字
\usepackage{mathtools}
\usepackage{amssymb}
\usepackage{amsfonts}
\usepackage[all,pdf]{xy}
\usepackage{bm}%添加带有粗体斜体的数学符号

\usepackage[bookmarks=true,colorlinks,citecolor=blue,linkcolor=black]{hyperref}
\usepackage{cleveref}
\usepackage{float}
\usepackage{caption}
\usepackage{subcaption}
\usepackage[margin=1in]{geometry}
\usepackage{graphicx}

\newtheorem{main}{Theorem}

\newtheorem{theorem}{Theorem}[section]
\newtheorem{definition}[theorem]{Definition}
\newtheorem{lemma}[theorem]{Lemma}
\newtheorem{proposition}[theorem]{Proposition}
\newtheorem{corollary}[theorem]{Corollary}

\newtheorem{remark}[theorem]{Remark}

\numberwithin{equation}{section}
\newcommand\thmref[1]{Theorem~\ref{#1}}
\newcommand\lemref[1]{Lemma~\ref{#1}}
\newcommand\propref[1]{Proposition~\ref{#1}}
\newcommand\corref[1]{Corollary~\ref{#1}}
\newcommand\defref[1]{Definition~\ref{#1}}
\newcommand\secref[1]{Section~\ref{#1}}
\newcommand\rmkref[1]{Remark~\ref{#1}}
\newcommand\figref[1]{Figure~\ref{#1}}

\newcommand{\Pexp}{P_\mathrm{exp}^\perp}
\newcommand{\eps}{\varepsilon}

\newcommand{\NN}{\mathbb{N}}
\newcommand{\RR}{\mathbb{R}}
\newcommand{\ZZ}{\mathbb{Z}}

\newcommand{\CCC}{\mathcal{C}}
\newcommand{\DDD}{\mathcal{D}}
\newcommand{\UUU}{\mathcal{U}}
\newcommand{\VVV}{\mathcal{V}}
\newcommand{\PPP}{\mathcal{P}}
\newcommand{\GGG}{\mathcal{G}}
\newcommand{\SSS}{\mathcal{S}}
\newcommand{\FFF}{\mathcal{F}}
\newcommand{\MMM}{\mathcal{M}}
\newcommand{\WWW}{\mathcal{W}}

\newcommand{\Diff}{\mathrm{Diff}}
\newcommand{\Emb}{\mathrm{Emb}}
\newcommand{\NE}{\mathrm{NE}}
\renewcommand{\top}{\mathrm{top}}
\newcommand{\dist}{\mathrm{dist}}

\begin{document}
\title[Unique equilibrium states]{Unique equilibrium states for some partially hyperbolic diffeomorphisms with dominated splittings}
\author{Qiao Liu and Jianxiang Liao}
\address{School of mathematics and computational science, Xiangtan university, Hunan, China}
%\address{}
\email{qiaoliu@xtu.edu.cn}
\email{202321511178@smail.xtu.edu.cn}
%\date{\today}
\begin{abstract}
    We prove robustness and uniqueness of equilibrium states for a class of partially hyperbolic diffeomorphisms with dominated splittings and H\"older continuous potentials with not very large oscillation.
\end{abstract}

\maketitle

\section{Introduction}
For a continuous map on compact metric space $f: M \rightarrow M$ and a continuous function $\phi: M \rightarrow \RR$, called a potential function, topological pressure of $(f,\phi)$ can be defined as, with variational principle \cite{walters1975variational},
\begin{equation*}
    P(f,\phi)=\sup_{\mu}\left\{h_\mu(f)+\int \phi d\mu\right\},
\end{equation*}
where the supremum is taken over all $f$-invariant Borel probability measures and $h_{\mu}(f)$ denotes the metric entropy of $f$ with respect to invariant measure $\mu$. An invariant probability measure that maximizes the quantity $h_\mu(f)+\int \phi d\mu$ is called equilibrium state. In particular, the topological pressure coincides with topological entropy when the potential function is identically zero. In this case, measure achieves the supremum is called measure of maximal entropy.

It has been a long-standing problem to find conditions that guarantee the existence and/or uniqueness of equilibrium states. The theory of uniqueness of equilibrium states of uniformly hyperbolic dynamical systems was developed by Sinai, Ruelle and Bowen in the mid-1970s. Sinai initiated the study \cite{sinai1972gibbs} in the case of Anosov diffeomorphism. Ruelle and Bowen generalized to uniformly hyperbolic (Axiom A) system \cite{bowen1971entropy, ruelle1968statistical, ruelle1978thermodynamic}. In \cite{bowen1974some}, Bowen provided a criterion to guarantee the existence and uniqueness of equilibrium states. It shows that $(M,f,\phi)$ has a unique equilibrium state whenever $(M,f)$ is an expansive system with specification and the continuous potential $\phi$ satisfies a certain regular condition ( known as the Bowen property). An important example that fits this criterion is the H\"older continuous potentials over uniformly hyperbolic systems.

For non-uniform hyperbolic dynamical systems, a recent breakthrough was made by Climenhaga and Thompson \cite{climenhaga2016unique}. Following Bowen's approach, they are able to weaken the three hypotheses to weak specification with all scale, expansivity observed by all measures with large free energy, and Bowen property on a suitable large collection of orbit segments. Applying this criterion, it is proved that geodesic flows in non-positive curvature admits a unique equilibrium state \cite{burns2018unique}, as do certain partially hyperbolic systems like the Bonatti-Viana family of diffeomorphism  and Ma\~{n}\'{e}'s derived from Anosov (DA) diffeomorphism on $\mathbb T^3$ \cite{climenhaga2018unique, climenhaga2019equilibrium}.
This so called Climenhaga-Thompson criterion was improved by Pacifico, Yan and Yang in \cite{pacifico2022existence} and applied to Lorenz-like attractors in any dimension to get uniqueness of equilibrium states \cite{pacifico2022uniqueness}.

In \cite{climenhaga2021beyond}, authors showed that partially hyperbolic systems with one-dimensional center which every equilibrium state has central Lyapunov exponent negative and the unstable foliation is minimal have a unique equilibrium state. Although the minimality condition implies $\eps$-minimality in some neighborhood $\UUU$. But the assumption of the central Lyapunov exponent is a not open condition, the robustness of the uniqueness of the equilibrium states does not follow directly from this work. To overcome this obstruction, Juan Carlos Mongez and Maria Jos\'e Pacifico \cite{mongez2025robustness} give a open condition in place of the negative Lyapunov exponents, that is $h^{u}(f)-h^{s}(f)> \sup\phi -\inf\phi\geq 0$, where $h^{u}(f)$ and $h^{s}(f)$ denote the unstable and stable entropy respectively which introduced in \cite{hu2017unstable, yang2021entropy}.

Following their work, we focus the robustness of the uniqueness of the equilibrium states of partially hyperbolic diffeomorphisms with dominated splittings which its central bundle $E^c$ can be decomposed into some one-dimensional subbundles.

\subsection*{Dominated Splittings}
Let $M$ be a compact Riemannian manifold, and $f:M\rightarrow M$ a $C^1$ diffeomorphism. We denote by $\|\cdot\|$ the norm induced by the Riemannian structure.

\begin{definition}
    We say that $f$ has a dominated splitting $TM=E\oplus F$ if
    \begin{itemize}
        \item[(i)] $E$ and $F$ are $Df$-invariant;
        \item[(ii)] The subbundles $E$ and $F$ are continuous, i.e. $E(x)$ and $F(x)$ vary continuously with $x\in M$;
        \item[(iii)] There exist $C>0$ and $0<\lambda<1$ such that
        \begin{equation*}
            \|Df^n |_{E(x)} \|\cdot \|Df^{-n} |_{F(f^n(x))} \| \leq C\lambda^n,
        \end{equation*}
        for all $x\in M$ and $n\geq 0$.
    \end{itemize}
\end{definition}
If $f$ has a dominated splitting $TM=E\oplus F$, we can denote it by $E\prec F$.

\begin{definition}{\rm (Plaque Family) }
    Let us consider a dominated splitting $TM=E\oplus F$. A \emph{plaque family} tangent to $E$ is a continuous map $\WWW$ from the linear bundle $E$ into $M$ such that
    \begin{itemize}
        \item for every $x\in M$, the map $\WWW_x:E_x\rightarrow M$ is a $C^1$-embedding which satisfies $\WWW_x(0)=x$ and whose image is tangent to $E_x$ at $x$;
        \item $(\WWW_x)_{x\in M}$ is a continuous family of $C^1$-embeddings.
    \end{itemize}
\end{definition}
The plaque family $\WWW$ is \emph{locally invariant} if there exists $\rho>0$ such that for every $x\in M$ the image of the ball $B(0,\rho)\subset E_x$ by $f\circ \WWW_x$ is contained in the plaque $\WWW_{f(x)}$.

We denote the image of $\WWW_x$ by itself. In \cite{hirsch1977invariant}, it is showed that the locally invariant plaque family tangent to $E$ always exists (not unique in general), see \thmref{thm:loc-palque}.

\begin{definition}{\rm (Trapped plaques)}\label{def:trapped-plaque}
    The plaque family $\WWW$ tangent to $E$ is \emph{trapped} if for each $x\in M$, one has
    \begin{equation*}
        f(\overline{\WWW_x})\subset \WWW_{f(x)}.
    \end{equation*}
    $\WWW$ is \emph{thin trapped} if for any neighborhood $S$ of the section $0$ in the $E$ there exists a continuous family $(\varphi_x)_{x\in M}$ of $C^1$-diffeomorphisms of the spaces $(E_x)_{x\in M}$ such that for every $x\in M$ one has
    \begin{equation*}
        \varphi_x(B(0,1))\subset S \, \mbox{ and } \, f(\overline{\WWW_x\circ \varphi_x(B(0,1))})\subset \WWW_{f(x)}\circ \varphi_{f(x)} (B(0,1)).
    \end{equation*}
\end{definition}

\begin{definition}{\rm (Thin Trapped)}
    The bundle $E$ is \emph{thin trapped} if every locally invariant plaque family tangent to $E$ is thin trapped, equivalently if one locally invariant plaque family tangent to $E$ is thin trapped (see \propref{prop:trapped-palque}).
\end{definition}

In general, a $Df$-invariant splitting of the form  $TM=E_1\oplus\cdots \oplus E_k$ if for every $i=1,\dots, k-1$ is dominated if for all $i=1,\dots,k$,  $TM=E_1^i \oplus E_{i+1}^k$ is dominated, where $E_l^m=E_l\oplus \cdots \oplus E_m$.

A diffeomorphism $f:M\rightarrow M$ is partially hyperbolic if it admits a $Df$-invariant dominated splitting
\begin{equation*}
    TM=E^s\oplus E^c\oplus E^u,
\end{equation*}
such that $E^s$ and $E^u$ are uniformly contracted by $Df$ and $Df^{-1}$ respectively.

We give a similar condition $h_{\top}(f)- \max\{h^u(f), h^s(f)\}> \sup\phi - \inf\phi\geq 0$ referenced from \cite{mongez2025robustness,mongez2025partially}. This condition guarantee that there are only non-zero central Lyapunov exponents for every ergodic equilibrium states. Our main result is the following:

\begin{main}\label{thm:main}
    Let $f: M\rightarrow M$ be a $C^{1+}$ partially hyperbolic diffeomorphism on a compact Riemannian manifold $M$ with a dominated splitting
    \begin{equation*}
        TM=E^s\oplus E_1^c \oplus E_2^c \oplus E^u,
    \end{equation*}
    where $\dim E_i^c = 1$ for $i=1, 2$, and $\phi:M\rightarrow \RR$ a H\"older continuous potential. If $h_{\top}(f)- \max\{h^u(f), h^s(f)\}> \sup\phi - \inf\phi\geq 0$, the bundle $E^s\oplus E_1^c$ is thin trapped, and the stable foliation $\FFF^s(f)$ is minimal, then there exists a $C^1$ neighborhood $\UUU$ of $f$ such that $g$ has a unique equilibrium states for all $g\in \UUU \cap \Diff^{1+}(M)$.
\end{main}

As a corollary of \thmref{thm:main}, we also have a similar result for the center consists of finite one-dimensional subbundles as following.

\begin{main}\label{thm:main+}
    Let $f: M\rightarrow M$ be a $C^{1+}$ partially hyperbolic diffeomorphism on a compact Riemannian manifold $M$ with a dominated splitting
    \begin{equation*}
        TM=E^s\oplus E_1^c \oplus\cdots \oplus E_k^c \oplus E^u,
    \end{equation*}
    where $\dim E_i^c = 1$ for $i=1,\dots, k$, and $\phi:M\rightarrow \RR$ a H\"older continuous potential. If $h_{\top}(f)- \max\{h^u(f), h^s(f)\}> \sup\phi - \inf\phi\geq 0$, the bundle $E^s\oplus E_1^c\oplus \cdots \oplus E_j^c$ ( for certain $j,\, 1 \le j \le k-1$) is thin trapped, and the stable foliation $\FFF^s(f)$ is minimal, then there exists a $C^1$ neighborhood $\UUU$ of $f$ such that $g$ has a unique equilibrium states for all $g\in \UUU \cap \Diff^{1+}(M)$.
\end{main}

\section{Preliminaries}
In this section, we give some relevant definitions.

\subsection{Topological pressure}
Let $M$ be a compact metric space, $f:M\rightarrow M$ a homeomorphism and $\phi:M\rightarrow \RR$ a continuous function. We denote the set of $f$-invariant Borel probability measures on $M$ by $\MMM_{f}(M)$. Given $n\geq:0, \delta >0$, and $x,y\in M$ we define the Bowen metric
\begin{equation}%\label{eq:zdn}
    d_n(x,y) \coloneqq \max_{0\leq i \leq n-1} d(f^i x, f^i y) ,
\end{equation}
and the Bowen balls
\begin{equation}%\label{eq:Bowen-ball}
    B_n(x,\delta) \coloneqq \{y\in X :\, d_n(x,y) < \delta\}.
\end{equation}

Given $\delta >0, n\geq 0$, we say a set $E\subset M$ is $(t,\delta)$-separated if for every distinct $x,y\in E$ we have $d_n(x,y)> \delta$. We write $X\times \NN$ as the space of finite orbit segments for $(X,f)$ by identifying $(x,n)$ with $(x,f(x),\dots, f^{n-1}(x))$. Given $\CCC\subset M\times \NN$ and $n\geq 0$ we write $\CCC_{n}=\{x\in M:\, (x,n)\in \CCC\}$.

Now fix a scale $\eps>0$, we write
\begin{equation*}
    \Phi_{\eps}(x,n) = \sup_{y\in B_{n}(x,\eps)} \sum_{k=0}^{n-1} \phi(f^k y).
\end{equation*}
In particular, $\Phi_0(x,n)= \sum_{k=0}^{n} \phi(f^k x)$. Given $\CCC\subset M\times \NN$ and $n\geq 0$, we consider the partition function
\begin{equation}\label{eq:part-func}
    \Lambda(\CCC, f, \phi, \delta, \eps, n)= \sup \left\{ \sum_{x\in E} e^{\Phi_{\eps}(x,n)} :\, E\subset \CCC_n \mbox{ is } (t,\delta)\mbox{-separated} \right\}.
\end{equation}
We call a $(n,\delta)$-separated set that attains the supremum in \eqref{eq:part-func} maximizing for $\Lambda(\CCC, f, \phi, \delta, \eps, n)$. When $\eps=0$, we simply write $\Lambda(\CCC, f, \phi, \delta, n)$ instead of $\Lambda(\CCC, f, \phi, \delta, 0, n)$.

The pressure of $\phi$ on $\CCC$ at scale $\delta, \eps$ is given by
\begin{equation*}
    P(\CCC,f,\phi,\delta,\eps)=\limsup_{n\rightarrow \infty} \frac1n \log \Lambda(\CCC,f,\phi,\delta,\eps,n).
\end{equation*}
It is clear that $\Lambda_n(\Phi,\delta,\eps)$ is monotonic in both $\delta$ and $\eps$, that is decrease for $\delta$ and increase for $\eps$. Thus same is true for $P$. Again, we write $P(\CCC, f, \phi, \delta)$ in place of $P(\CCC, f, \phi, \delta, 0)$, and let
\begin{equation*}
    P(\CCC, f, \phi)=\lim_{\delta\rightarrow 0} P(\CCC, f, \phi, \delta),
\end{equation*}
since the function $P(\CCC, f, \phi,\delta)$ is non-decreasing as $\delta\rightarrow 0$, so that the limit in the right hand of above equation exists. When $\CCC=M\times \NN$, it is the notion of topological pressure of $\phi$ with respect to $f$ denoted by $P(f,\phi)$.

Recall that $\MMM_f(M)$ is the space of $f$-invariant Borel probability measures on $M$. Let $\MMM_{f}^{e}$ be the set of ergodic measures in $\MMM_f(M)$. The classical variational principle shows that
\begin{equation*}
    P(f,\phi) = \sup_{\mu\in \MMM_f(M)} \left\{ h_{\mu}(f) + \int \phi \,d\mu \right\} =\sup_{\mu\in \MMM_f^e(M)} \left\{ h_{\mu}(f) + \int \phi \,d\mu \right\}
\end{equation*}
An invariant measure $\mu\in \MMM_f(X)$ is called an equilibrium state of $\phi$ if $\mu$ attain this supremum. When $\phi=0$, the equilibrium states is called measures of maximal entropy. For a invariant measure $\mu$, we define the pressure of $\mu$ as
\begin{equation*}
    P_{\mu}(f,\phi)\coloneqq h_{\mu}(f) + \int \phi \,d\mu.
\end{equation*}

\subsection{Criteria for uniqueness of equilibrium states}
In this section, We will introduce a criteria developed by \cite{climenhaga2016unique, pacifico2022existence} concerning the existence of a unique equilibrium state.

\begin{definition}
    A decomposition for $\DDD \subset M\times \NN$ consists of three collections $(\PPP,\GGG,\SSS) \subset M\times \NN$ and three functions $p,g,s: \DDD \rightarrow \NN$ such that for every $(x,n) \in \DDD$, the values $p=p(x,n), g=g(x,n)$, and $s=s(x,n)$ satisfy $n=p+g+s$, and
    \begin{equation*}
        (x,p)\in \PPP,\enspace (f^px, g)\in \GGG, \enspace (f^{p+g}x,s)\in \SSS.
    \end{equation*}
\end{definition}
If $\DDD= M\times \NN$, we say that $(\PPP,\GGG,\SSS)$ is a decomposition for $(M,f)$.

\subsubsection{Specification}
The specification play a crucial role to obtain the uniqueness of equilibrium states or measures of maximal entropy.

\begin{definition}
    A collection of orbit segments $\GGG\subset M\times \NN$ has weak specification at scale $\delta$ if there exist $\tau \in \NN$ such that for every $\{(x_i,n_i)\}_{i=1}^k \subset \GGG$ there exist a point $y$ and a sequence of ``gluing times'' $\tau_1,\dots, \tau_{k-1}\in \NN$ with $\tau_i\leq \tau$ such that for $N_j=\sum_{i=1}^{j} n_i + \sum_{i=1}^{j-1} \tau_i$ and $N_0=\tau_0=0$, we have
    \begin{equation*}
        d_{n_i}(f^{N_{j-1}+\tau_{j-1}}y, x_j)< \delta \mbox{ for every } 1\leq j\leq k.
    \end{equation*}
    We say that $\GGG\subset M\times\NN$ has weak specification if it has weak specification at every scale $\delta>0$.
\end{definition}

Intuitively, the ``specification" property say that for any finite orbit segments in $\GGG$, there exists a point $y\in M$ which $\delta$-shadows each orbit segment with gaps of length $\tau_i$ between each segment and the gap $\tau_i$ has a uniform bound $\tau$.

\begin{remark}
    Because weak specification is the only specification property which we consider in this paper. So that, we henceforth use the term specification for simplicity.
\end{remark}

\subsubsection{The Bowen property}
The Bowen property tell us some potentials satisfying certain bounded distortion property is what we expect in order to obtain uniqueness of equilibrium states. It was first introduced by Bowen \cite{bowen1974some}.

\begin{definition}\label{def:bowen-property}
    Given $\GGG\subset M\times \NN$, a continuous potential $\phi$ has the Bowen property on $\GGG$ at scale $\eps$ if there exists $K>0$ so that
    \begin{equation*}
        \sup\{ | \Phi_0(x,n)- \Phi_0(y,n) | :\, (x,n)\in \GGG, y\in B_n(x,\eps)  \}\leq K.
    \end{equation*}
    We say that $\phi$ has the Bowen property on $\GGG$ if there exists $\eps>0$ such that $\phi$ has the Bowen property on $\GGG$ at scale $\eps$.
\end{definition}
We call such $K$ the distortion constant as in \defref{def:bowen-property}.

\subsubsection{Expansivity}
Given $f: M\rightarrow M$ a homeomorphism. For $x\in M$ and $\eps>0$, let us consider the bi-infinite Bowen ball
\begin{equation*}
    \Gamma_{\eps}(x) = \{ y\in M :\, d(f^n x,f^n y)\leq \eps \mbox{ for all } n\in \ZZ \}.
\end{equation*}
Note that $\Gamma_{\eps}=\bigcap_{n\in \ZZ} f^{n} \overline{B}_{2n}(x,\eps)$ is compact for every $x,\eps$.

We say that $f$ is expansive if there exists a $\eps>0$ such that $\Gamma_{\eps}(x)=\{x\}$ for every $x\in M$. For a non-expansive system we can consider the set of non-expansive points at scale $\eps$
\begin{equation}
    \NE(\eps)=\bigl\{x\in X :\, \Gamma_{\eps}(x)\neq \{x\}\bigr\},
\end{equation}
it is clear that $\NE(\eps)$ is a $f$-invariant set. We say a $f$-invariant measure $\mu$ is almost expansive at scale\, $\eps$ \,if\, $\mu(\NE(\eps))=0$.

\begin{definition}
    For a potential $\phi$, the pressure of obstructions to expansivity at scale $\eps$ is defined by
    \begin{equation*}
        \begin{split}
            \Pexp(f,\phi,\eps)&=\sup_{\mu\in \MMM_{f}^e(M)}\left\{h_{\mu}(f) + \int \phi \, d\mu :\, \mu(\NE (\eps))=0 \right\},\\
            &=\sup_{\mu\in \MMM_{f}^e(M)}\left\{h_{\mu}(f) + \int \phi \, d\mu :\, \mu(\NE (\eps))=1 \right\}.
        \end{split}
    \end{equation*}
    And scale-free quantity is
    \begin{equation*}
        \Pexp(f,\phi)=\lim_{\eps\rightarrow 0}\Pexp(f,\phi,\eps).
    \end{equation*}
\end{definition}

Note that $\Pexp(f,\phi,\eps)$ is non-increasing as $\eps\rightarrow 0$, so that the limit in the above definition exists.

We say that $f$ is entropy expansive or $h$-expansive if there exists a $\eps>0$ such that $h_{\top}(\Gamma_{\eps}(x))=0$ for every $x\in M$.

\subsubsection{A criteria to uniqueness of equilibrium states}
Now we provide a criteria to uniqueness of equilibrium states from \cite{climenhaga2016unique}. We cite here an improved version of this criteria in \cite{pacifico2022existence} in the following.

\begin{theorem}{\rm (\cite[Theorem A]{pacifico2022existence})}\label{thm:ct-criteria}
    Let $f:M\rightarrow M$ be a homeomorphism on a compact metric space and $\phi: M\rightarrow \RR$ a continuous potential. Suppose there are $\eps$ and $\delta$ with $\eps>2000\delta$ such that $\Pexp(f,\phi,\eps)< P(f,\phi)$ and $\DDD\subset M\times \NN$ which admits a decomposition $(\PPP, \GGG, \SSS)$ with the following properties:
    \begin{itemize}
        \item[(1)] $\GGG$ has weak specification at scale $\delta$;
        \item[(2)] $\phi$ has the Bowen property at scale $\eps$ on $\GGG$;
        \item[(3)] $P(\DDD^{c}\cup \PPP\cup \SSS, f, \phi,\delta,\eps)< P(f,\phi)$.
    \end{itemize}
    Then there exists a unique equilibrium state for the potential $\phi$.
\end{theorem}

\subsubsection{Existence of equilibrium states}
Let $f:M\rightarrow M$ be a partially hyperbolic diffeomorphism and $\phi:M\rightarrow \RR$ a continuous potential. It is difficult to prove that some partially hyperbolic system has equilibrium states. However, in sometimes we can solve it. Recall that the pressure of $\mu$ is $P_{\mu}(f,\phi)= h_{\mu}(f) + \int \phi \,d\mu$. Since $\MMM_f(M)$ is a compact space, if the pressure is upper semicontinuous, then it can achieves the supremum in variational principle, so $(f,\phi)$ has an equilibrium state.

For homeomorphisms, Bowen proved \cite{bowen1972entropy} that the metric entropy is upper semicontinuous whenever $(M,f)$ is entropy expansive. In particular, the map $P_{\mu}(f,\phi)$ is upper semicontinuous, so that $(f,\phi)$ has an equilibrium state for any continuous potentials. For partially hyperbolic systems which the central bundle is decomposed into finite many one-dimensional subbundles, it is showed that such systems are entropy expansive \cite{Lorenzo2012entropy, liao2013entropy, arbieto2025existence}, and thereby they have equilibrium states.

\subsection{Partial hyperbolicity with Dominated Splittings}
For partially hyperbolic diffeomorphisms, it is well known that there are unique foliations $\FFF^{\sigma}$ tangent to the stable and unstable distributions $E^{\sigma}$ for $\sigma=s,u$ \cite{hirsch1977invariant}.These foliations are called stable and unstable foliations respectively. The leaf of $\FFF^{\sigma}$ containing $x$ corresponds to the stable or unstable manifold $W^{\sigma}(x)$, for $\sigma=s,u$.

In the case of dominated splitting, we need a structure so-called locally invariant plaque families introduced in \cite[Theorem 5.5]{hirsch1977invariant}. We cite a version of the plaque family theorem due to \cite[Lemma 3.5]{crovisier2015essential}.

\begin{theorem} {\rm (\cite{hirsch1977invariant,crovisier2015essential})} \label{thm:loc-palque}
    Suppose that a diffeomorphism $f:M\rightarrow M$ admits a dominated splitting $TM=E\oplus F$. Then there exist a $C^1$ neighborhood $\UUU$ of $f$, a number $\rho >0$, and a continuous family of embeddings
    \begin{equation*}
        \UUU\times M \ni (g,x) \mapsto \WWW_{g,x} \in \Emb^1(\RR^{\dim E},M)
    \end{equation*}
    such that for every $g\in\UUU$ and every $x\in M$ we have
    \begin{itemize}
        \item $T_x \WWW_{g,x}(\RR^{\dim E})=E_g(x)$, and
        \item $g(\WWW_{g,x}(B_{\rho}(0)))\subset \WWW_{g,g(x)}(\RR^{\dim E})$.
    \end{itemize}
\end{theorem}

For a partially hyperbolic diffeomorphism $f$ with dominated splitting $TM=E^s\oplus E_1^c\oplus E_2^c\oplus E^u$, we write $E^{cs}=E^s\oplus E_1^c$ and $E^{cu}=E_2^c\oplus E^u$. Since $f$ has dominated splitting $E^{cs}\prec E^{cu}$, according to \thmref{thm:loc-palque} there is a $C^1$ neighborhood $\UUU$ of $f$ such that for every $g\in \UUU$ there exists a locally invariant plaque families $\WWW_g^{cs}$ tangent to $E_g^{cs}$. And the inverse of $f$ has a dominated splitting $E^{cu}\prec E^{cs}$, so that there is a $C^1$ neighborhood $\UUU$ of $f$ such that for every $g\in \UUU$ there exists a locally invariant plaque families $\WWW_g^{cu}$ tangent to $E_g^{cu}$.

\begin{remark}\label{rmk:subfoliation}
    The uniqueness of stable and unstable foliations implies, via a standard argument {\rm (\cite[Theorem 6.1(e)]{hirsch1977invariant})}, the $W_{loc}^s$ subfoliates $\WWW^{cs}$, and the $W_{loc}^u$ subfoliates $\WWW^{cu}$ for some suitable number $loc$.
\end{remark}

Recall the thin trapped plaque families as in \defref{def:trapped-plaque}. This notion allows us to construct a small trapped plaque family and plays an important role in the proof of specification property.

\begin{proposition}{\rm (\cite[Lemma 2.8]{crovisier2015essential})} \label{prop:trapped-palque}
    Let $TM=E\oplus F$ be a dominated splitting over $M$ such that there exists a thin trapped plaque family $\WWW$ tangent to $E$. Then any other locally invariant plaque family tangent to $E$ is thin trapped.
\end{proposition}

\begin{proposition}{\rm (\cite[Lemma 3.6]{crovisier2015essential})} \label{prop:robust thin-trapped}
    Let $TM=E\oplus F$ be a dominated splitting over $M$, and $E$ is thin trapped. Let $(\widetilde{\WWW}_g)$ be a continuous and uniformly locally invariant collection of plaque families tangent to the bundles $(E_g)$ for $g$ $C^1$-close to $f$.

    Then for every $\rho>0$, there exist $C^1$ neighborhood $\UUU$ of $f$ and a continuous collection of plaque families $(\WWW_g)_{g\in \UUU}$ tangent to the bundles $(E_g)$, which is trapped by $g$ and whose plaques have diameter smaller than $\rho$. Moreover we have $\WWW_g(x)\subset\widetilde{\WWW}_g(x)$ for every $g\in \UUU$ and $x\in M$.
\end{proposition}

\subsection{Minimal foliation}
For the foliation $\FFF^{\sigma}, (\sigma=s,u)$, we introduce the notation of minimal foliation as following.
\begin{definition}
    Let $f:M\rightarrow M$ be a partially hyperbolic diffeomorphism. The foliation $\FFF^{\sigma}$ is minimal if $\FFF^{\sigma}(x)$ is dense in $M$ for all $x\in M$, $(\sigma=s,u)$.
\end{definition}

\begin{definition}
    Let $f:M\rightarrow M$ be a partially hyperbolic diffeomorphism. The stable foliation $\FFF_f^{s}$ of $f$ is $\eps$-minimal if there exists $R>0$ such that if $D$ is a disk contained in a stable leaf of $\FFF_f^{s}$ with a internal radius large than $D$ is $\eps$-dense in $M$.
\end{definition}

It is well known that if $f : M \rightarrow M$ is a partially hyperbolic diffeomorphism whose stable foliation is minimal, then, for every $\eps> 0$, there exists a $C^1$ neighborhood $\UUU$ of $f$ such that for any $g \in \UUU$, the stable foliation is $\eps$-minimal \cite[Lemma 8.1]{climenhaga2018unique}.

\subsection{Lyapunov exponents}
Let $f:M\rightarrow M$ be a diffeomorphism on a compact Riemannian manifold $M$. For a point $x\in M$ and nonzero vector $v\in T_xM$, the Lyapunov exponent of $f$ at $x$ in direction $v$ is defined as
\begin{equation*}
    \lambda(f,x,v)=\lim_{n\rightarrow \infty} \frac{1}{n} \log\| Df^n(x) \cdot v \|,
\end{equation*}
if the limit of right-hand side exists.

Let $\mu\in\MMM_f(M)$, the Oseledets theorem guarantee the existence of Lyapunov exponents for $\mu$-almost $x\in M$.
\begin{theorem}{\rm (\cite{oseledets1968multiplicative})}
    Let $f:M\rightarrow M$ be a $C^1$ diffeomorphism on a compact Riemannian manifold $M$ and $\mu$ a ergodic $f$-invariant measure. Then for $\mu$-almost $x\in M$, there exist
    \begin{itemize}
        \item real numbers $\lambda_1(f,\mu)>\dots>\lambda_k(f,\mu)$ $(k\leq n)$,
        \item a splitting
        \begin{equation*}
            T_xM=E_{1}(x)\oplus \cdots \oplus E_{k}(x),
        \end{equation*}
    \end{itemize}
    such that $Df_xE_i(x)=E_i(f(x))$ and the limit
    \begin{equation*}
        \lambda_i(f,x)\coloneqq \lim_{n\rightarrow \infty} \frac1n \log\|Df^n(x)\cdot v\|=\lim_{n\rightarrow -\infty} \frac1n \log\|Df^n(x)\cdot v\|=\lambda_i,
    \end{equation*}
    exist for all $v\in E_i(x)\setminus \{0\}$ and $\lambda_i(f,x)=\lambda_i(f,\mu)$.
\end{theorem}

\begin{definition}
    An ergodic measure $\mu$ is hyperbolic if it has no zero Lyapunov exponents, that is $|\lambda_i(f,\mu)|>0$ for every $i$.
\end{definition}

\subsection{Unstable entropy}
In this section, we recall the notation of unstable metric entropy and unstable topological entropy for partially hyperbolic diffeomorphisms introduced in \cite{hu2017unstable}.

Let $f$ be a partially hyperbolic diffeomorphism. For a partition $\alpha$ of $M$, let $\alpha(x)$ denote the element of $\alpha$ that contains $x$. A partition $\alpha$ is finer than a partition $\beta$, which we denote by $\alpha \geq \beta  \mbox{ or } \beta \leq \alpha$, if every element $\alpha(x)$ is contained in $\beta(x)$ for all $x\in M$.

\begin{definition}
    We say a partition $\PPP$ is measurable (or countably generated) with respect to measure $\mu$ if there exists a measurable family $\{ A_i\}_{i\in\NN}$ and a measurable set $F$ of full measure such that if $B\in \PPP$, then there exists a sequence $\{ B_i\}$, where $B_i\in \{A_i,A_i^c\}$ such that $B\cap F=\bigcap_{i}B_i\cap F$.
\end{definition}

A measurable partition $\xi$ is called increasing if $f^{-1}\xi \geq \xi$. For a measurable partition $\beta$, we denote $\beta_m^{n}=\vee_{i=m}^{n}f^{-i}\beta$.

Take $\eps_0$ small. Let $\PPP=\PPP_{\eps_0}$ denote the set of finite measurable partitions whose elements have a diameter smaller than or equal to $\eps_0$. For each $\beta \in \PPP$ we can define a finer partition $\eta$ such that $\eta(x)=\beta(x)\cap W_{loc}^{u}(x)$ for each $x\in M$, where $W_{loc}^{u}(x)$ represents a local unstable manifold at $x$ with size greater than the diameter $\eps_0$ of $\beta$. Let $\PPP^{u}=\PPP^{u}_{\eps_o}$ denote the set of such partitions.

A measurable partition $\alpha$ is said to be subordinate to unstable manifolds of $f$ with respect to a measure $\mu$ if for $\mu$-almost every point $x$, $\alpha(x)\subset W^u(x)$ and contains an open neighborhood of $x$ in $W^u(x)$. It is not difficult to check if measurable partition $\alpha$ with $\mu(\partial\alpha)=0$ where $\partial \alpha\coloneqq\bigcup_{A\in \alpha}\partial A$, then the corresponding partition $\eta$ in $\PPP^u$ given by $\eta(x)=\alpha(x)\cap W_{loc}^u(x)$ is a partition subordinated to unstable manifolds of $f$.

For measurable $\eta$ partition of probability measure space $(M,\mu)$, the canonical system of condition measures for $\mu$ is a family of probability measure $\{\mu_x^{\eta}: x\in M \}$ a disintegration of $\mu$ such that $\mu^{\eta}_x(\eta(x))=1 $ and for every measurable set $B\subset M$, the map $x\mapsto \mu_{x}^{\eta}(B)$ is measurable and $\mu(B)=\int_M \mu_{x}^{\eta}(B)d\mu(x)$. (See e.g. \cite{rokhlin1952fundamental} for reference.)

Given two measurable partition $\alpha, \eta$ of measure space $(M,\mu)$, the conditional entropy of $\alpha$ given $\eta$ with respect to $\mu$ is defined as
\begin{equation*}
    H_{\mu}(\alpha|\eta)=-\int_M \log \mu_x^{\eta}(\alpha(x))d\mu(x).
\end{equation*}

\begin{definition}
    The conditional entropy of $f$ with respect to a measurable partition $\alpha$ given $\eta$ is defined as
    \begin{equation*}
        h_\mu(f, \alpha|\eta)=\limsup_{n\rightarrow \infty}{\frac{1}{n}H_{\mu}(\alpha^{n-1}_{0}|\eta)}.
    \end{equation*}
    The conditional entropy of $f$ given $\eta \in \PPP^u $ is defined as
    \begin{equation*}
        h_\mu(f|\eta)=\sup_{\alpha \in \PPP}h_\mu(f,\alpha|\eta),
    \end{equation*}
    and the unstable metric entropy of $f$ is defined as
    \begin{equation*}
        h_\mu^u(f)=\sum_{\eta\in\PPP^u}h_\mu(f|\eta).
    \end{equation*}
\end{definition}

Now we introduce the notion of unstable topological entropy. Let $d^u$ denote the metric induced by the Riemannian structure on the unstable manifold and $d^u_n(x,y)=\max_{0\leq i\leq n-1}\{d^u(f^i(x), f^i(y))\}$. A subset is called $(n,\eps)$-separated if all elements are pairwise $\eps$ $d^u_n$-distance away. Consider open ball $B^u(x,\delta)$ inside the unstable manifold $W^u(x)$, let $N^u(f,\eps,n,x,\delta)$ denote the maximal cardinality of $(n,\eps)$-separated subsets of $\overline{B^u(x,\delta)}$.

\begin{definition}
    The unstable topological entropy of $f$ on M is defined as
    $$h^u_{\top}(f)=\lim_{\delta\rightarrow 0}\sum_{x\in M} h^u_{\top}(f,\overline{B^u(x,\delta)})$$
    wehre
    $$h^u_{\top}(f,\overline{B^u(x,\delta)})=\lim_{\eps \rightarrow 0}\limsup_{n\rightarrow \infty}\frac{1}{n}N^u(f,\eps,n,x,\delta).$$
\end{definition}

Just like the usual metric entropy and topological entropy, the unstable metric entropy and unstable topological entropy can be related through a variational principle. Indeed, the following holds \cite[Theorem D]{hu2017unstable}.
\begin{equation*}
    h^u_{\top}(f)=\sup\{h^u_\mu(f):\, \mu \in \MMM_f(M)\}=\sup\{h^u_\nu(f):\, \nu \in \MMM_f^e(M)\}.
\end{equation*}

The stable metric entropy and the stable topological entropy can be defined simply as $h^s_\mu(f)=h^u_\mu(f^{-1})$ and $h^s_{\top}(f)=h^u_{\top}(f^{-1})$. The variational principle also relates them.

\section{Proof of \thmref{thm:main}}\label{sec:main-proof}
Let $f: M\rightarrow M$ be a $C^{1+}$ partially hyperbolic diffeomorphism on a compact $M$ with a dominated splitting
    \begin{equation*}
        TM=E^s\oplus E_1^c \oplus E_2^c \oplus E^u,
    \end{equation*}
where $\dim E_i^c = 1$ for $i=1, 2$, and $\Diff^{1+}(M)$ be the collection of $C^{1+}$ diffeomorphisms on $M$. Suppose that $\dim E^s$ and $\dim E^u$ are greater than $0$, and $\phi$ is a H\"older continuous potential. Assume that $0<\max\{ \| Df|_{E^s} \|,\| Df^{-1}|_{E^u} \| \}< \xi<1$.

Now we consider the inequality
\begin{equation}\label{eq:ent-ineq}
    h_{\top}(f)- \max\{h^u(f), h^s(f)\}> \sup\phi - \inf\phi\geq 0.
\end{equation}
In fact, this assumption plays a crucial role in proving that the integrated central Lyapunov exponents $\lambda_i^c(f,\mu)$ $(i=1,2)$  is away from zero. In the following section, we will prove that the inequality \eqref{eq:ent-ineq} is open in $C^{1+}$ topology. Throughout this section, we assume the inequality \eqref{eq:ent-ineq} always holds.

\begin{remark}
    Let us comment on this hypothesis. A related condition, $P_{\top}(f,\phi)>\sup \phi$, was introduced by Denker, Urba\'nski \cite{denker1991ergodic}, in the context of rational maps on the sphere. Another related condition, $\sup\phi-\inf \phi < h(f)-\log q$, is used by Varandas, Viana \cite{varandas2010existence}, in the context of non-uniformly expanding maps. Our condition seems to play a similar role in our setting. 
\end{remark}

\subsection{Open property of \eqref{eq:ent-ineq}}
For given $f\in \Diff^{1+}(M)$ and $\mu$ ergodic measure,  Ledrappier and Young entropy formula says that
\begin{equation}\label{eq:u-ent-formula}
    h_{\mu}(f)\leq h^{u}_{\mu}(f)+ \sum_{\lambda^c(f,\mu)>0} \lambda^c(f,\mu),
\end{equation}
By replacing $f$ with $f^{-1}$, we also obtain
\begin{equation*}
    h_{\mu}(f)\leq h^{s}_{\mu}(f)- \sum_{\lambda^c(f,\mu)<0} \lambda^c(f,\mu).
\end{equation*}
As a consequence of this formula, if the central Lyapunov exponents $\lambda^{c}(f,\mu)$ are non-positive (respectively non-negative) then
\begin{equation}\label{eq:us-ent}
    h_{\mu}(f)=h_{\mu}^u(f) \, (\mbox{respectively } h_{\mu}(f)=h_{\mu}^s(f)).
\end{equation}
\begin{remark}
    The above formula \eqref{eq:u-ent-formula} of Ledrappier and Young was proved in $C^2$ regularity \cite{ledrappier1985metric1, ledrappier1985metric2}. And Brown improved the formula to $C^{1+\alpha}$ regularity \cite{brown2022smoothness}.
\end{remark}

\begin{lemma}\label{lem:hyp-meas}
    Let $\mu$ be an ergodic equilibrium state for $(f,\phi)$. If \eqref{eq:ent-ineq} holds, then $\lambda_1^c(f,\mu)< 0$ and $\lambda_2^c(f,\mu)> 0$.
\end{lemma}
\begin{proof}
    Let $\mu$ be an ergodic equilibrium state for $(f, \phi)$, assuming $\lambda_1^c(f,\mu)$ and $\lambda_2^c(f,\mu)$ are both non-negative. By \eqref{eq:us-ent}, we have
    \begin{equation*}
        P_{\mu}(f,\phi)=h_{\mu}(f)+\int \phi\,d\mu=h_{\mu}^s(f) + \int \phi\,d\mu\leq h^s(f)+\sup\phi.
    \end{equation*}
    Observe that \eqref{eq:ent-ineq} implies that there is an invariant measure $\nu$ such that
    \begin{equation*}
        h_{\nu}(f)+ \inf\phi > h^s(f)+ \sup\phi,
    \end{equation*}
    hence
    \begin{equation*}
        P_{\mu}(f,\phi)< h_{\nu}+ \inf\phi \leq P(f,\phi),
    \end{equation*}
    which contradicts that $\mu$ is an ergodic equilibrium state. Therefore, $\lambda_1^c(f,\mu)$ and $\lambda_2^c(f,\mu)$ can't be both non-negative. Similarly, we can prove that $\lambda_1^c(f,\mu)$ and $\lambda_2^c(f,\mu)$ can't be both non-positive which yields $\lambda_1^c(f,\phi)<0$ and $\lambda_2^c(f,\phi)>0$.
\end{proof}

\begin{lemma}\label{lem:posi-ent}
    Let $\mu$ be an ergodic equilibrium state for $(f,\phi)$, then $h_{\mu}(f)>0$.
\end{lemma}
\begin{proof}
    Let $\nu$ be a measure of maximal entropy, then we have
    \begin{equation*}
        P_{\nu}(f,\phi)=h_{\nu}(f)+\int\phi\,d\nu\geq h_{\top}(f)+ \inf\phi> h^u(f)+ \sup\phi\geq \sup\phi.
    \end{equation*}
    Also we have $h_{\mu}(f)+\int\phi\,d\mu=P(f,\phi)\geq P_{\nu}(f,\phi)>\sup\phi$. This implies $h_{\mu}(f)>\sup\phi-\int\phi\,d\mu\geq 0$.
\end{proof}

\lemref{lem:posi-ent} guarantees that every ergodic equilibrium state has positive metric entropy. And \lemref{lem:hyp-meas} shows that all such measures are hyperbolic measure which allows us to apply Katok's horseshoe \cite[Theorem S.5.9]{katok1995introduction} to approximate the free energy $P_{\mu}(f,\phi)$ arbitrarily close with topological pressure $P(f |_{\Lambda},\phi |_{\Lambda})$ restricted on a hyperbolic horseshoe $\Lambda$.

\begin{proposition}\label{prop:cont-press}
    The map $g\rightarrow P(g,\phi)$ with $g\in \Diff^{1+}(M)$ is continuous in $C^1$ topology.
\end{proposition}
\begin{proof}
    The upper semi-continuity is given by \cite[Theorem 4.3]{mongez2025partially}. The lower semi-continuity is from katok's horseshoe \cite{katok1995introduction} or \cite{gelfert2016horseshoes}. Also, the proof can be found in \cite[Lemma 3.4]{mongez2025robustness}.
\end{proof}

\begin{corollary}\label{cor:open-ent-ineq}
    There exists a $C^1$ neighborhood $\VVV$ of $f$, such that for every $g\in \VVV\cap \Diff^{1+}(M)$, we have
    \begin{equation}\label{eq:open-ent-ineq}
        h_{\top}(g)- \max\{h^u(g), h^s(g)\}> \sup\phi - \inf\phi\geq 0.
    \end{equation}
\end{corollary}
\begin{proof}
    We conclude that the map $g\rightarrow h_{\top}(g)$ is continuous in $C^1$ topology. Indeed, let us consider the zero potential $\phi=0$, We have $P(g,\phi)=h_{\top}(g)$ and condition \eqref{eq:ent-ineq} $h_{\top}(f)- \max\{h^u(f), h^s(f)\}> 0$ is satisfied. Hence the map $g\rightarrow P(g,0)$ is continuous by \propref{prop:cont-press}.

    Given $\eps >0$, because of continuity of map $g\rightarrow h_{\top}(g)$, there is a $C^1$ neighborhood $\VVV_1$ of $f$ such that
    \begin{equation*}
        | h_{\top}(f)-h_{\top}(g) |< \eps,
    \end{equation*}
    for every $g\in \VVV_1 \cap \Diff^{1+}(M)$. And we also have upper semi-continuity of maps $h^u,h^s: \Diff^{1+}(M)\rightarrow \RR$, so that there is a $C^1$ neighborhood $\VVV_2$ of $f$ such that
    \begin{equation*}
       \max\{ h^u(g), h^s(g) \}<\max\{ h^u(f),h^s(f) \}+\eps,
    \end{equation*}
    for every $g\in \VVV_2\cap \Diff^{1+}(M)$. Let $0<\eps <\frac12 (h_{\top}(f)+\inf\phi-\max\{h^u(f),h^s(f)\}-\sup\phi)$, and $\VVV=\VVV_1\cap\VVV_2$, then we have \eqref{eq:open-ent-ineq} for every $g\in \VVV\cap \Diff^{1+}(M)$.
\end{proof}

\begin{remark}
    Without loss of generality, we can assume $\|Dg|_{E^s}\|, \| Dg^{-1}|_{E^u}\|< \xi<1$ for every $g\in \VVV$.
\end{remark}

\subsection{Construction of the $(\PPP,\GGG,\SSS)$ decomposition}
In this section, we give a decomposition $(\PPP_g,\GGG_g,\SSS_g)$ for every orbit segment $(x,n)\in M\times \NN$ in such a way that $\GGG_g$ captures almost all ``hyperbolicity'' . We follow some ideas from \cite{climenhaga2021beyond,mongez2025robustness,burns2018unique}. In their work, the `good' collection of orbit segment $\GGG_g$ starts at `hyperbolic' time such that $\GGG_g$ obtain almost hyperbolicity. However, they deal with the setting of one-dimensional center, so that it can be decomposed into hyperbolic time and non-hyperbolic time in the center direction, that is $(\PPP_g,\GGG_g)$. In our setting of two-dimensional center, we decompose every orbit segment $(x,n)$ into three part:
\begin{itemize}
    \item $\PPP_g$ : non-hyperbolicity in $E_1^c$ direction;
    \item $\GGG_g$ : hyperbolicity in $E_1^c$ and $E_2^c$ directions;
    \item $\SSS_g$ : non-hyperbolicity in $E_2^c$ direction.
\end{itemize}

Let $\VVV$ be a $C^1$ neighborhood of $f$ as in \corref{cor:open-ent-ineq}. For every $g\in \VVV\cap \Diff^{1+}(M)$, let us denote the central Lyapunov exponents by $\lambda_1^c(g,\mu)$ and $\lambda_2^c(g,\mu)$. Now consider the quantities
\begin{align*}
    P_1^+(g,\phi)&=\sup\{P_{\mu}(g,\phi):\, \mu\in \MMM_f^e(M), \lambda_1^c(g,\mu)\geq 0\},\\
    P_1^-(g,\phi)&=\sup\{P_{\mu}(g,\phi):\, \mu\in \MMM_f^e(M), \lambda_1^c(g,\mu)\leq 0\},\\
\end{align*}
and
\begin{align*}
    P_2^+(g,\phi)&=\sup\{P_{\mu}(g,\phi):\, \mu\in \MMM_f^e(M), \lambda_2^c(g,\mu)\geq 0\},\\
    P_2^-(g,\phi)&=\sup\{P_{\mu}(g,\phi):\, \mu\in \MMM_f^e(M), \lambda_2^c(g,\mu)\leq 0\}.
\end{align*}
\lemref{lem:hyp-meas} and \corref{cor:open-ent-ineq} tell us that under our assumption there is no ergodic equilibrium state such that $\lambda_1^c(g,\mu)\geq 0$ and $\lambda_2^c(g,\mu)\leq 0$. Since $P(g,\phi)=\max\{P_1^+(g,\phi),P_1^-(g,\phi)\}= \max\{P_2^+(g,\phi),P_2^-(g,\phi)\}$, with upper semi-continuity of pressure, it implies that
\begin{equation*}
    P_1^+(g,\phi)< P(g,\phi) \enspace \mathrm{and}\enspace P_2^-(g,\phi)< P(g,\phi).
\end{equation*}
By ergodic decomposition theorem, we get
\begin{equation*}
    \sup\{P_{\mu}(g,\phi):\, \mu\in \MMM_f(M), \lambda_1^c(g,\mu)\geq 0\}< P(g,\phi),
\end{equation*}
and
\begin{equation*}
    \sup\{P_{\mu}(g,\phi):\, \mu\in \MMM_f(M), \lambda_2^c(g,\mu)\leq 0\}< P(g,\phi),
\end{equation*}
for every $g\in \VVV\cap\Diff^{1+}(M)$.

\begin{lemma}\label{lem:less-press}
    There exist $\eta >0$, $a>0$, and a $C^1$ neighborhood $\widetilde{\VVV} \subset \VVV$ such that for every $g\in \widetilde{\VVV} \cap\Diff^{1+}(M)$, we have
    \begin{equation*}
        \sup\{P_{\mu}(g,\phi):\, \mu\in \MMM_g(M), \lambda_1^c(g,\mu)\geq -\eta \}<a< P(g,\phi),
    \end{equation*}
    and
    \begin{equation}\label{eq:E_2-less-press}
        \sup\{P_{\mu}(g,\phi):\, \mu\in \MMM_g(M), \lambda_2^c(g,\mu)\leq \eta \}<a< P(g,\phi).
    \end{equation}
\end{lemma}
\begin{proof}
    The proof is similar to Lemma 4.1 and Corollary 4.2 in \cite{mongez2025robustness}.
\end{proof}

\begin{remark}\label{rmk:inverse-lya-exponent}
    In following discussion, we rewrite the condition in \eqref{eq:E_2-less-press}. It is clear that $\lambda_2^c(g,\mu)\leq \eta$ is equivalent to $\lambda_2^c(g^{-1},\mu)\geq -\eta$.
\end{remark}

Finally, we can construct a decomposition for every $g$ which is close to $f$ in the following way. Fix $r>0$, $a>0$ and $\widetilde{\VVV}\subset\Diff^{1}(M)$ as in \lemref{lem:less-press}. For every $g\in \widetilde{\VVV}\cap \Diff^{1+}(M)$, let $\varphi_1^c(x)=\log \|Dg |_{E_1^c(x)}\|$ and $\widehat{\varphi}_2^c(x)=\log \|Dg^{-1} |_{E_2^c(x)}\|$, for given $\mu\in \MMM_g(M)=\MMM_{g^{-1}}(M)$ we have the central Lyapunov exponents
\begin{equation*}
    \begin{split}
        \lambda_1^c(g,x)&=\lim_{n\rightarrow \infty} \frac{1}{n} S_n \varphi_1^c(x),\\
        \lambda_2^c(g^{-1},x)&=\lim_{n\rightarrow \infty} \frac{1}{n} S_n^{-1} \widehat{\varphi}_2^c(x),
    \end{split}
\end{equation*}
for $\mu$-almost everywhere, where $S_n\varphi_1^c(x)=\sum_{k=0}^{n-1}\varphi_1^c(g^kx)$ and $S_n^{-1}\widehat{\varphi}_2^c(x)=\sum_{k=0}^{n-1}\widehat{\varphi}_2^c(g^{-k}x)$. Moreover, if $\mu$ is ergodic then $\lambda_1^c(g,\mu)=\int \varphi_1^c (x)\, d\mu$ and $\lambda_2^c(g^{-1},\mu)=\int \widehat{\varphi}_2^c (x)\, d\mu$.

For $r>0$, we define
\begin{equation}\label{eq:P_g}
    \PPP_g(r)\coloneqq \{ (x,n)\in M\times \NN:\, S_n\varphi_1^c(x)\geq -rn\}.
\end{equation}
Let us describe the `good' collection of orbit segment $\GGG_g$ and `bad' collection of orbit segment $\SSS_g$. To do this, let us consider an arbitrary orbit segment $(x,n)\in M\times\NN$. We remove the longest possible element of $\PPP_g$ from its beginning. That is, let $p = p(x,n)$ be maximal with the property that $(x,p) \in \PPP_g$. Then we have
\begin{equation*}
    S_p\varphi_1^c(x)\geq -rp, \mbox{ and } S_k\varphi_1^c(x)< -rk, \mbox{ for all } p<k\leq n.
\end{equation*}
It gives
\begin{equation*}
    S_{k-p}\varphi_1^c(g^px)=S_k\varphi_1^c(x)-S_p\varphi_1^c(x)<-r(k-p),
\end{equation*}
which we can rewrite as
\begin{equation}\label{eq:P_g-eq}
    S_j\varphi_1^c(x)<-rj, \mbox{ for all } 0\leq j\leq n-p,
\end{equation}

To define $\SSS_g(r)$, we notice that $(M,g)$ and $(M,g^{-1})$ have same orbit segment space (ignoring the direction of the orbits). We write orbit segment $(x,n)$ backwards so that it become the orbit segment of $g^{-1}$,
\begin{equation*}
    g^{n-1}x, \, \dots, \, gx, \, x.
\end{equation*}
Similarly, we let $s=s(x,n)$ be maximal with properties that $S_s^{-1}\widehat{\varphi}_2^c(g^{n-1}x)\geq -rs$ and $s\leq n-p$.

\begin{equation*}
    S_s^{-1} \widehat{\varphi}_2^c(g^{n-1}x)\geq -rs, \mbox{ and } S_k^{-1}\widehat{\varphi}_2^c(g^{n-1}x)< -rk, \mbox{ for all } s<k\leq n-p,
\end{equation*}
then we have
\begin{equation}\label{eq:S_g-trans}
    S_{k-s}^{-1}\widehat{\varphi}_2^c(g^{n-s-1}x)=S_k^{-1}\widehat{\varphi}_2^c(g^{n-1}x)-S_s^{-1} \widehat{\varphi}_2^c(g^{n-1}x)< -r(k-s).
\end{equation}
Rewrite \eqref{eq:S_g-trans} as
\begin{equation}\label{eq:S_g-eq}
    S_{j}^{-1}\widehat{\varphi}_2^c(g^{j-1}x)<-rj,  \mbox{ for all } 0\leq j\leq n-p-s.
\end{equation}
In summary, combining \eqref{eq:P_g}, \eqref{eq:P_g-eq} and \eqref{eq:S_g-eq}, we have a decomposition for $(x,n)$. They are characterized by
\begin{equation}\label{eq:PGS}
    \begin{aligned}
        \PPP_g(r) & = \{ (x,n)\in M\times \NN:\, S_n\varphi_1^c(x)\geq -rn\},\\
        \GGG_g(r) & = \{ (x,n)\in M\times \NN:\, S_k\varphi_1^c(x)< -rk \mbox{ and }  S_{k}^{-1}\widehat{\varphi}_2^c(g^{k-1}x)<-rk, \, \forall \, 0\leq k\leq n\},\\
        \SSS_g(r) & = \{ (x,n)\in M\times \NN:\, S_n^{-1} \widehat{\varphi}_2^c(g^{n-1}x)\geq -rn\} \setminus \PPP_g.
        %& = \{(x,n)\in M\times \NN:\, S_n\varphi_2^c(x)\leq rn\} \setminus \PPP_g,
    \end{aligned}
\end{equation}
where $n$ is the maximal integer satisfying \eqref{eq:P_g-eq} and \eqref{eq:S_g-eq}.
%where $\varphi_2^c(x)=\log\| Dg |_{E_2^c(x)} \|$ and we have $\lambda_2^c(g,\mu)=\int \varphi_2^c(x)\,d\mu$ with ergodic $\mu$.

\subsection{Specification on $\GGG_g(r)$}
In this section, we will prove the Specification on $\GGG_g$. We first give a theorem which is close to \cite[Theorem 3.5]{ANDERSSON20201245} and \cite[Theorem 4.3]{mongez2025robustness}.

Consider the diffeomorphism $f:M\rightarrow M$ with dominated splitting $TM=E^s\oplus E_1^c\oplus E_2^c\oplus E^u$. We denote $E^{cs}=E^s\oplus E_1^c$ and $E^{cu}=E_2^c\oplus E^u$. For simplicity we write $\WWW_g^{cs}(x)= \WWW_{g,x}^{cs}(\RR^{\dim E^{cs}})$ and $\WWW_g^{cu}(x)= \WWW_{g,x}^{cu}(\RR^{\dim E^{cu}})$, where $\WWW_{g,x}^{cs}$ and $\WWW_{g,x}^{cu}$ as in \thmref{thm:loc-palque}.
\begin{theorem}\label{thm:loc-disks}
    There exist a $C^1$ neighborhood $\widetilde{\UUU}$ of $f$, a constants $\rho=\rho(f)>0$, $C=C(f)>0$ and $\eps^{\prime}>0$ such that for every $g\in \widetilde{\UUU}$, we have a continuous collection of trapped plaque families $\widetilde{\WWW}_{g}^{cs}$ tangent to $E_g^{cs}$ with diameter smaller than $\rho$.

    Moreover, if $(x,n)\in \GGG_g(r)$, then the following statements holds:
    \begin{itemize}
        \item[(1)] There are two $C^1$ embedded disks $W_{\rho}^{cs}(g,x)$ and $W_{\rho}^{cu}(g,x)$ which dimensions equal to $\dim E^{cs}$ and $\dim E^{cu}$ respectively, and radius $\rho>0$, centered at $x$, such that
        \begin{equation*}
            T_x W_{\rho}^{cs}(g,x)=E^{cs}(g,x) \mbox{ and } T_x W_{\rho}^{cu}(g,x)=E^{cu}(g,x);
        \end{equation*}
        \item[(2)] $W_{\rho}^{cs}(g,x)$ and $W_{\rho}^{cu}(g,x)$ depend continuously on both $x$ and $g$ in the $C^1$ topology;
        \item[(3)] If $y\in W_{\rho}^{cs}(g,x)$ and $z\in W_{\rho}^{cu}(g,x)$, then
        \begin{equation*}
            d(g^kx,g^ky)\leq C e^{-kr/2} d(x,y) \mbox{ and } d(g^{n-k}x,g^{n-k}z)\leq C e^{-kr/2} d(g^n x,g^n z),
        \end{equation*}
        for every $0\leq k \leq n$;
        \item[(4)] $\widetilde{\WWW}_{g}^{cs}(x)\subset W_{\rho}^{cs}(g,x)$ for every $g\in \widetilde{\UUU}$;
        \item[(5)] If $d(x,y)< \eps^{\prime}$, then
        \begin{equation}\label{eq:LPS}
            \widetilde{\WWW}_{g}^{cs}(y) \cap W_{\rho}^{cu}(g,x) \neq \varnothing,
        \end{equation}
        where $y$ is not necessary in some $(\GGG_g(r))_n$. Furthermore, if $z\in \widetilde{\WWW}_{g}^{cs}(y) \cap W_{\rho}^{cu}(g,x)$, then $d(x,z) \leq C d(x,y)$.
    \end{itemize}
\end{theorem}
\begin{proof}
    The existence of trapped plaque families $\widetilde{\WWW}_g^{cs}$ can be obtained directly from \propref{prop:robust thin-trapped}.

    Now suppose that $(x,n)\in \GGG_g(r)$, let us consider the dominated splitting $E^{cu}\prec E^{cs}$ of inverse of $f$, by \thmref{thm:loc-palque} there is a $C^1$ neighborhood $\widetilde{\UUU}$ of $f$, a number $\rho^{\prime} >0$, and a continuous family of embeddings $\WWW_{g,x}^{cu} \in \Emb^1(\RR^{\dim E^{cu}},M)$ such that for every $g\in \widetilde{\UUU}$ and every $x\in M$ we have
    \begin{itemize}
        \item $T_x \WWW_{g,x}^{cu}(\RR^{\dim E^{cu}})=E^{cu}_{g}(x)$
        \item $g^{-1}(\WWW_{g,x}^{cu} (B_{\rho^{\prime}}(0)))\subset \WWW_{g,g^{-1}(x)}^{cu}(\RR^{\dim E^{cu}})$.
    \end{itemize}

    Let $\WWW_g^{cu}(x)=\WWW_{g,x}^{cu}(\RR^{\dim E^{cu}})$ and  let $D_g(x,c)=\{y\in \WWW_g^{cu}(x): \dist_{\WWW_g^{cu}(x)}(x,y) <c\}$ be a disk of radius $c$ where $\dist_{\WWW_g^{cu}(x)}$ denote the distance given by the Riemannian structure in $\WWW_g^{cu}(x)$.

    Take $c_0$ small enough such that  $D_g(x,c_0)\subset \WWW_{g,x}^{cu}(B_{\rho}(0))$ for every $(g,x) \in \widetilde{\UUU} \times M$, then we have $g^{-1}(D_g(x,c_0))\subset \WWW_g^{cu}(g^{-1}(x))$. In particular, for any $c_1 \le c_0$, there exists $c_2 \le c_1$ such that $g^{-1}(D_g(x, c_2))\subset D_g(g^{-1}(x),c_1)$ for every $(g,x) \in \widetilde{\UUU} \times M$.

    By uniform continuity of $\log \|Df^{-1}|_{E_g^{cu}(x)}\|$, upon possibly reducing $\widetilde{\UUU}$, we may assume that there exists $c_1 < c_0$ such that $\|Dg^{-1}|_{E_g^{cu}(y)}\| \leq e^{r/2} \|Dg^{-1}|_{E_g^{cu}(x)}\|$ wherever $g\in \widetilde{\UUU}$ and $d(x,y)<c_1$.

    We claim that for every $y\in D_g(x,c_1)$ and every $(x,n)\in \GGG_g(r)$ we have
    \begin{equation*}
        \prod_{i=0}^{k} \|Dg^{-1}|_{E_g^{cu}(g^{n-i}(y))}\| \leq e^{-kr/2}
    \end{equation*}
    for every $0\le k \le n$. This can be obtained through induction and the definition of $\GGG_g(r)$. When $k=0$ it follows directly from $\| Dg^{-1}|_{E_g^{cu}(y)}\| \leq e^{r/2} \|Dg^{-1}|_{E_g^{cu}(x)}\| \leq  e^{r/2} e^{-r} \leq  e^{-r/2}  $. Assume that it holds for $i<k$. $\dist_{\WWW_g^{cu}}(g^{n-i}(x), g^{n-i}(y))\leq e^{-ir/2} \dist_{\WWW_g^{cu}}(x,y) \leq c_1$ for every $i < k$,  hence $g^{n-k+1}(y) \in D_g(g^{n-k+1}(x),c_1)$ which implies
    \begin{equation*}
        \begin{split}
            \dist_{\WWW_g^{cu}}(g^{n-k}(x), g^{n-k}(y)) &= \dist_{\WWW_g^{cu}}(g^{-1}(g^{n-k+1}(x), g^{-1}(g^{n-k+1}(y))))\\
            &\leq e^{-r/2} \dist_{\WWW_g^{cu}}(g^{n-k+1}(x), g^{n-k+1}(y)) \leq c_1.
        \end{split}
    \end{equation*}
    Therefore we have $\prod_{i=0}^{k} \|Dg^{-1}|_{E_g^{cu}(g^{n-i}(y))}\| \leq e^{kr/2}\prod_{i=0}^{k} \|Dg^{-1}|_{E_g^{cu}(g^{n-i}(x))}\| \leq e^{-kr/2}$ as desired.

    Let $W_{\rho}^{cu}(g,x)=D_g(x,\rho)$ for sufficiently small $\rho < c_2 $, we shall also assume that $\rho$ is small enough so that each $D_g(x,\rho)$ is uniformly transversal to $E^{cs}$ for every $(g,x)$,  then both (1), (2) are true for $W_{\rho}^{cu}(g,x)$. Item (3) is given by our induction process.

    Similarly we can consider dominated splitting $E^{cs}\prec E^{cu}$ to obtain $W_{\rho}^{cs}(g,x)$ for every $g\in \widetilde{\UUU}$ and every $x\in M$ satisfying item (1)(2)(3).

    Item (4) holds since $\widetilde{\WWW}_g^{cs}(x)\subset \WWW_g^{cs}(x)$ and we can reduce the diameter of $\widetilde{\WWW}_g^{cs}(x)$ if necessary. Item (5) follows from uniform transversality and continuity of $W_{\rho}^{cu}(g,x)$ and $W_{\rho}^{cs}(g,y)$.
\end{proof}

\begin{remark}\label{rmk:Wcu}
    Applying \thmref{thm:loc-disks} to $g^{-1}$, if $(x,n)\in \GGG_g(r)$, then there exists a $C^1$ embedded disk $W_{\rho}^{cu}(g,g^n x)$ tangent to $E_g(g^n x)$ such that
    \begin{itemize}
        \item $W_{\rho}^{cu}(g,g^n x)$ is continuous;
        \item If $g^n z\in W_{\rho}^{cu}(g,g^n x)$, then
        \begin{equation*}
            d(g^{n-k}x,g^{n-k}z)\leq C e^{-kr/2} d(g^n x,g^n z), \, \mbox{ for every } \, 0\leq k\leq n.
        \end{equation*}
        \item If $d(g^n x, y)<\eps^{\prime}$, then
        \begin{equation*}
             \widetilde{\WWW}_{g}^{cs}(y) \cap W_{\rho}^{cu}(g,g^n x) \neq \varnothing,
        \end{equation*}
        with $d(g^n x, z)\leq C d(g^n x, y)$ where $z\in \widetilde{\WWW}_{g}^{cs}(y) \cap W_{\rho}^{cu}(g,g^n x)$.
    \end{itemize}
\end{remark}

\begin{remark}\label{rmk:LPS}
    If $y\in \GGG_g(r)$ in item (5) of \thmref{thm:loc-disks} , then we can rewrite \eqref{eq:LPS} as following formula
    \begin{equation*}
        W_{\rho}^{cs}(g,x) \cap W_{\rho}^{cu}(g,y) \neq \varnothing.
    \end{equation*}
    Furthermore, if $z\in W_{\rho}^{cs}(g,x) \cap W_{\rho}^{cu}(g,y)$, then $d(x,z) \leq C d(x,y)$.
\end{remark}

Let $(x,n)\in \GGG_g(r)$ and $y$ be a contained in a Bowen ball of $x$. We have no idea about information which like item (3) of $y$. To overcome this obstruction, we can use uniform continuity of $\varphi_1^c$ and $\widehat{\varphi}_2^c$. Given $r>0$, let $\widehat{\eps} =\widehat{\eps}(r) >0$ be small enough so that for any $x,y \in M$, we have
\begin{equation}\label{eq:uni-con}
    d(x,y) < \widehat{\eps} \Rightarrow | \varphi_1^c(x)-\varphi_1^c(y) | \leq \frac{r}{2} \mbox{ and } | \widehat{\varphi}_2^c(x)-\widehat{\varphi}_2^c(y) | \leq \frac{r}{2}.
\end{equation}
In particular, if $x\in \GGG_g(r)$ and $y\in B_n(x,\widehat{\eps})$, then
\begin{equation*}
    S_k \varphi_1^c(y) \leq S_k \varphi_1^c(x) + rk/2 < -rk/2 , \mbox{ for every } 0\leq k \leq n.
\end{equation*}
and
\begin{equation*}
    S_k^{-1} \widehat{\varphi}_2^c(y) \leq S_k^{-1} \widehat{\varphi}_2^c(x) + rk/2 < -rk/2 , \mbox{ for every } 0\leq k \leq n.
\end{equation*}
Therefore we have $(y,n) \in \GGG_g(r/2)$.

\begin{proposition}\label{prop:specification}
    Let $f$ has minimal stable foliation and $0<\delta<\widehat{\eps}$, there exists a $C^1$ neighborhood $\UUU$ of $f$ such that if $g\in \UUU\cap \Diff^{1+}(M)$ then g has specification on $\GGG_g(r)$ at scale $\delta>0$.
\end{proposition}
\begin{proof}
    Let $\widetilde{\UUU}$ and $\eps^{\prime}$ be as in \thmref{thm:loc-disks}. Fix $0<\eps< \eps^{\prime}$, and $(x_1,n_1), (x_2,n_2)\in \GGG_g$. Since $\FFF_f^s$ is minimal, then there is a $C^1$ neighborhood $\UUU_{\eps}$ of $f$ such that for every $g\in \UUU_{\eps}$ has $\eps$-minimal stable foliation, that is, there exists $R=R(g)$ such that if $D$ is a disk contained in a stable leaf of $g$ with an internal radius larger than $R/2$, then $D\cap B_{\eps}(x)\neq \varnothing$ for any $x\in M$.

    Let $\UUU=\widetilde{\UUU}\cap \UUU_{\eps}$. For the above $R>0$ and $g\in \UUU$, we fix a $N$ such that if $x,y$ are contained in a stable disck with radius less than $R$, then $d(g^{N} x,g^{N}y)\leq  \delta/2$.

    Let $D$ be a disk contained in $W^s(g^{-N} x_2)$ centered at $g^{-N} x_2$ with radius $R/2$. Since $\eps$-minimality of stable foliation, we have $D\cap B_{\eps}(g^{n_1} x_1)\neq \varnothing$. Pick a point $\widehat{y}\in D\cap B_{\eps}(g^{n_1} x_1)$, by \rmkref{rmk:Wcu}, we can choose a point $y\in \widetilde{\WWW}_g^{cs}(\widehat{y}) \cap W_{\rho}^{cu}(g,g^{n_1} x_1) $ with $d(y,\widehat{y})\leq C\eps$ and $d(y,g^{n_1}x_1)\leq C\eps$ by choosing $\eps > 0$ enough small. Since $\widetilde{\WWW}_g^{cs}$ is trapped by $g$, we have $g^{N}y\in \widetilde{\WWW}_g^{cs}(g^{N}\widehat{y})$, see \figref{fig:fig1}.

    \begin{figure}[ht]
        \centering
        \includegraphics[width=1\textwidth]{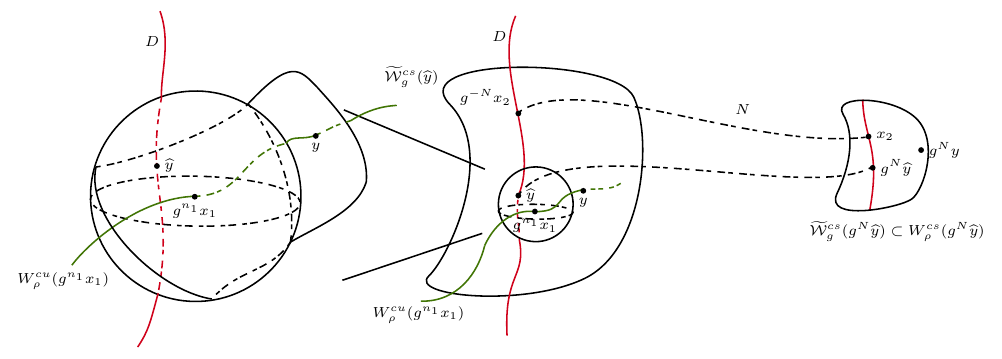}
        \caption{Proving the Specification. We enlarge the detail of ball $B_{\eps}(g^{n_1}x_1)$ on the left-hand side.}
        \label{fig:fig1}
    \end{figure}

    On the one hand, $\widehat{y}\in W^s(g^{-N}x_2)$, we have
    \begin{equation}\label{eq:u-contract}
        d(g^j x_2, g^j(g^{N} \widehat{y}))\leq \frac12 \delta \xi^{j} \mbox{ for every }  0\leq j\leq n_2.
    \end{equation}
    It showed that $g^{N} (\widehat{y})$ stays in Bowen ball $B_{n_2}(x_2,\widehat{\eps})$. Thus $(g^{N} (\widehat{y}), n_2) \in \GGG_g(r/2)$. Also $g^{N} y\in \widetilde{\WWW}_g^{cs}(g^{N}\widehat{y}) \subset W_{\rho}^{cs}(g^{N} \widehat{y})$ by item (4) in \thmref{thm:loc-disks} and, if necessary, reducing the diameter of plaque family $(\widetilde{\WWW}_g^{cs})$ such that $d(g^{N} y, g^{N} \widehat{y})\leq \delta/2$, it implies
    \begin{equation}\label{eq:cu-contract}
        d(g^j (g^{N}\widehat{y}), g^{j}(g^{N} y))\leq \frac12 \delta e^{-j r/4} \mbox{ for every } 0 \leq j \leq n_2.
    \end{equation}
    Combining \eqref{eq:u-contract} with \eqref{eq:cu-contract} yields
    \begin{equation*}
        d(g^j x_2, g^{j}(g^{N} y)) \leq \frac12 \delta (\xi^{j} + e^{-jr/4}) \mbox{ for every } 0 \leq j \leq n_1.
    \end{equation*}

    On the other hand, since $y\in W_{\rho}^{cu}(g,g^{n_1} x_1)$, and  by \rmkref{rmk:Wcu}, it implies
    \begin{equation*}
        d_{n_2}(x_2,g^{-n_1} y)\leq C^2 \eps.
    \end{equation*}

    Therefore, choosing $\tau = N$, the point $g^{-n_1}y$ is what we need. It establishes the ``one-step specification" property, by \cite[lemma 1.2.4.1]{climenhaga2021beyond}, it gives that $\GGG_g(r)$ has specification at scale $\eps K$ where $K$ is a positive constant that can be choosing uniformly with $g$. For given $\delta>0$, we can choose a sufficiently small $\eps$ such that $\eps K \leq \delta$.
\end{proof}

\subsection{The Bowen Property}
We will show that H\"older continuous potentials $\phi$ on $M$ have the Bowen property on $\GGG_g(r)$ for every $g\in \UUU$.

\begin{proposition}\label{prop:bowen}
    Let $\UUU$ and $\eps^{\prime}$ as in \propref{prop:specification} and \thmref{thm:loc-disks} respectively. If $g \in \UUU \cap \Diff^{1+}(M)$ and potential $\phi: M\rightarrow \RR$ is H\"older continuous then $\phi$ has the Bowen property on $\GGG_g(r)$.
\end{proposition}
\begin{proof}
    Let $g\in \UUU \cap \Diff^{1+}(M)$ and $0< \eps <\eps^{\prime},\widehat{\eps}$, such that \eqref{eq:uni-con} holds. For every $(x,n)\in \GGG_g(r)$ and $y\in B_n(x,\eps)$, then $(y,n) \in \GGG_g(r/2)$. By \rmkref{rmk:LPS}, we can choose a point $[x,y]\in W_{\rho}^{cs}(g,x) \cap W_{\rho}^{cu}(g,y)$ such that, by item (3) in \thmref{thm:loc-disks},
    \begin{equation*}
        d(g^k x,g^k [x,y])\leq Ce^{-kr/2}\delta_0 \mbox{ for all } k=0,1,\dots,n-1,
    \end{equation*}
    and
    \begin{equation*}
        d(g^k y,g^k [x,y])\leq Ce^{-(n-k)r/4}\delta_0 \mbox{ for all } k=0,1,\dots,n-1,
    \end{equation*}
    where $\delta_0$ is a number such that $d(z, [z,\hat{z}])$ and $d(\hat{z}, [z,\hat{z}])$ are both less than $\delta_0$ for every $z, \hat{z} \in \GGG_g(r)$ with $d(z,\hat{z})< \eps^{\prime}$. By the triangle inequality, we have
    \begin{equation*}
        d(g^k x, g^k y)\leq C \delta_0 (e^{-kr/2} + e^{-(n-k)r/4})
    \end{equation*}

    Let $Q$ be the H\"older constant and $\alpha$ the H\"older exponent, then
    \begin{equation*}
        \begin{split}
            | \phi(g^k x) - \phi(g^k y) |&\leq Q d(g^k x, g^k y)^{\alpha}\\
            &\leq Q (C \delta_0)^{\alpha} (e^{-kr/2} + e^{-(n-k)r/4})^{\alpha}.
        \end{split}
    \end{equation*}
    for every $0\leq k \leq n-1$. It follows that
    \begin{equation*}
        \begin{split}
            | \Phi_0(x,n)- \Phi_0(y,n) | &\leq \sum_{k=0}^{n-1} (C\delta_0)^{\alpha} Q (e^{-kr/2} + e^{-(n-k)r/4})^{\alpha}\\
            &\leq (C\delta_0)^{\alpha} Q \sum_{k=0}^{n-1} (e^{-kr/2} + e^{-(n-k)r/4})^{\alpha}\\
            &\leq (C\delta_0)^{\alpha} Q \sum_{k=0}^{\infty} (e^{-kr/2} + e^{-kr/4})^{\alpha} \eqcolon K.
        \end{split}
    \end{equation*}
    Since $K$ is finite and independent of $x, y, n$, this finishes the proof.
\end{proof}

\subsection{Pressure estimates of obstructions of expansivity and ``bad" orbit segments}
We take $r=\eta$ and $g$ as in \lemref{lem:less-press}. Let's consider the decomposition $(\PPP_g(\eta), \GGG_g(\eta), \SSS_g(\eta))$.

\begin{proposition}\label{prop:expansivity}
    There exist $\eps_1 > 0$ and $C^{1}$ neighborhood $\UUU$ of $f$ such that if $g \in \UUU$, then $\Pexp(g,\phi,\eps_1) < P(g, \phi)$.
\end{proposition}
\begin{proof}
    In \cite[Theorem 3.1]{liao2013entropy} the authors proved that there exist a $C^1$ neighborhood $\UUU$ and some constant $\eps_1 > 0$ such that $h(g, \Gamma_{\eps_1}(g,x)) = 0$ for every $g \in \UUU$, every $g$-invariant probability $\mu$, and $\mu$-almost every $x \in M$. More specifically, it shows that if $\mu \in \MMM_g^e$ does not have zero Lyapunov exponents, then we have $\Gamma_{\eps_1} (g,x) = \{x\}$ for $\mu$-almost every $x \in M$. It means that the collection of ergodic measures which are not almost expansive is contained in the collection of ergodic measures with zero central Lyapunov exponent. Hence, the \lemref{lem:less-press} implies that $\Pexp(g,\phi,\eps_1) < P(g,\phi)$ for every $g\in \UUU$.
\end{proof}

\begin{remark}
    In \cite{liao2013entropy} it showed this result for $f\in \Diff^{1}(M)\setminus \overline{\mathrm{HT}^{1}(M)}$ which is far away from tangencies. Note that if $f$ is a partially hyperbolic diffeomorphism with a dominated splitting $E^{s}\oplus E_1^c\oplus \cdots \oplus E_k^c\oplus E^{u}$, where $\dim E_i^c = 1$, then $f\in \Diff^{1}(M)\setminus \overline{\mathrm{HT}^{1}(M)}$.
\end{remark}

Now let us prove the orbit segments $\PPP_g(\eta) \cup \SSS_g(\eta)$ has less pressure than $X\times \NN$ which it is showed by the following result. Let $\widetilde{\VVV}$ be a $C^1$ neighborhood of $f$ as in \lemref{lem:less-press}.
\begin{proposition}\label{prop:bad set}
    There exists $\eps_2 > 0$ such that $P(\PPP_g(\eta) \cup \SSS_g(\eta), g, \phi, \delta, \eps) < P(g, \phi)$ for every $0 < \eps < \eps_1$ and $g \in \widetilde{\VVV}$.
\end{proposition}
\begin{proof}
    Let $E_n \subset (\PPP_g(\eta) \cup \SSS_g(\eta))_n$ be a maximizing $(n,\delta)$-separated set with $\delta>0$ for $\Lambda(\PPP_g(\eta) \cup \SSS_g(\eta), g, \phi, \delta, n)$, i.e. $\Lambda(\PPP_g(\eta) \cup \SSS_g(\eta), g, \phi, \delta, n) = \sum_{x\in E_n} e^{\Phi_0(x,n)}$. Then consider the measures
    \begin{equation*}
        \begin{split}
            \nu_n &\coloneqq \frac{\sum_{x\in E_n} e^{\Phi_0(x,n)}\delta_x}{\sum_{x\in E_n} e^{\Phi_0(x,n)}},\\
            \mu_n &\coloneqq \frac1n \sum_{i=0}^{n-1} g_{*}^i \nu_{i}.
        \end{split}
    \end{equation*}
    By compactness there is a sequence $n_k \rightarrow \infty$ such that $\mu_{n_k}$ converges in the weak* topology. Let $\mu=\lim_{k} \mu_{n_k}$, then $\mu$ is $g$-invariant. By the second part of the proof of \cite[Theorem 9.10]{walters2000introduction}, we have $P(\PPP_g(\eta)\cup \SSS_g(\eta), g, \phi, \delta)\leq P_{\mu}(g,\phi)$. Since $\PPP_g(\eta)$ and $\SSS_g(\eta)$ are characterized by \eqref{eq:PGS}, we have $\lambda_1^c(\mu)\geq -\eta$ or $\lambda_2^c(\mu)\leq \eta$. Hence
    \begin{equation*}
        P(\PPP_g(\eta)\cup \SSS_g(\eta), g, \phi, \delta)< a< P(g,\phi).
    \end{equation*}

    By the uniformly continuity of $\phi$, there exists $\eps_2$ such that
    \begin{equation*}
        P(\PPP_g(\eta)\cup \SSS_g(\eta), g, \phi, \delta, \eps)< a< P(g,\phi), \mbox{ for all } 0 < \eps < \eps_2 \mbox{ and } g\in \widetilde{\VVV}.
    \end{equation*}
    This finishes the proof.
\end{proof}

\subsection{Proof of \thmref{thm:main}}
Let $f:M\rightarrow M$ be a partially hyperbolic diffeomorphism with a dominated splitting $E^{s}\oplus E_1^c \oplus E_2^c \oplus E^{u}$ and $\phi:M\rightarrow \RR$ a H\"older continuous potential.

Let $\UUU$ be a $C^1$ neighborhood of $f$ satisfying the above discussion. We choose a $\eps$ such that $0<\eps < \eps^{\prime}, \eps_1,\eps_2$, where $\eps^{\prime}, \eps_1$ and $\eps_2$ as in \propref{prop:bowen}, \propref{prop:expansivity} and \propref{prop:bad set}, respectively. For every $g\in \UUU\cap \Diff^{1+}(M)$, there exists a decomposition $(\PPP_g(\eta), \GGG_g(\eta), \SSS_g(\eta))$ for $X\times \NN$ such that $\GGG_g(\eta)$ has the Bowen property at scale $\eps$ and
\begin{equation*}
    \begin{split}
        &\Pexp(g,\phi,\eps) < P(g,\phi),\\
        &P(\PPP_g(\eta)\cup \SSS_g(\eta), g, \phi, \delta, \eps)< P(g,\phi),
    \end{split}
\end{equation*}
for every $g\in \UUU\cap \Diff^{1+}(M)$ and $\delta > 0$. Applying \propref{prop:specification}, we obtain that $g$ has specification on $\GGG_g(\eta)$ at scale $\delta$ for every $g\in \UUU\cap \Diff^{1+}(M)$ with $\delta< 2000\eps$. Therefore, by \thmref{thm:ct-criteria}, $(g,\phi)$ has a unique equilibrium state.

\section{Proof of \thmref{thm:main+}}
The proof is similar to that of \thmref{thm:main}. We give a sketch proof which can be adapted from the arguments in \secref{sec:main-proof}.

Let $f: M\rightarrow M$ be a $C^{1+}$ partially hyperbolic diffeomorphism on a compact $M$ with a dominated splitting
    \begin{equation*}
        TM=E^s\oplus E_1^c \oplus \cdots \oplus E_k^c \oplus E^u,
    \end{equation*}
where $\dim E_i^c = 1$ for $i=1, \dots, k$. And we have the inequality \eqref{eq:ent-ineq}:
\begin{equation*}
    h_{\top}(f)- \max\{h^u(f), h^s(f)\}> \sup\phi - \inf\phi\geq 0.
\end{equation*}

Let $\lambda_i^c(f,\mu)$ be the Lyapunov exponent corresponding to $E_i^c$ for $i=1,\dots,k$, respectively.

\begin{lemma}
    Let $\mu$ be an ergodic equilibrium state for $(f,\phi)$. If \eqref{eq:ent-ineq} holds, then there exist a $i\in \{1,\dots,k-1\}$ such that
    \begin{itemize}
        \item $\lambda_j^c(f,\mu)< 0$, if $j\leq i$;
        \item $\lambda_j^c(f,\mu)> 0$, if $j\geq i+1$.
    \end{itemize}
\end{lemma}

\begin{proposition}
    There exists a $C^1$ neighborhood $\VVV$ of $f$, such that for every $g\in \VVV\cap \Diff^{1+}(M)$, we have
    \begin{equation*}
        h_{\top}(g)- \max\{h^u(g), h^s(g)\}> \sup\phi - \inf\phi\geq 0.
    \end{equation*}
\end{proposition}

Let $\gamma_1(g,\mu)=\max \{ \lambda_1^c(g,\mu), \dots, \lambda_i^c(g,\mu) \}$ and $\gamma_2(g,\mu)=\min \{ \lambda_{i+1}^c(g,\mu), \dots, \lambda_k^c(g,\mu) \}$. Then we have
\begin{proposition}\label{lem:less-press+}
    There exist $\eta >0$, $a>0$, and a $C^1$ neighborhood $\widetilde{\VVV} \subset \VVV$ such that for every $g\in \widetilde{\VVV} \cap\Diff^{1+}(M)$, we have
    \begin{equation*}
        \sup\{P_{\mu}(g,\phi):\, \mu\in \MMM_g(M), \gamma_1(g,\mu)\geq -\eta \}<a< P(g,\phi),
    \end{equation*}
    and
    \begin{equation*}
        \sup\{P_{\mu}(g,\phi):\, \mu\in \MMM_g(M), \gamma_2(g,\mu)\leq \eta \}<a< P(g,\phi).
    \end{equation*}
\end{proposition}
It is clear that $\gamma_2(g,\mu)\leq \eta$ is equivalent to $\gamma_2(g^{-1},\mu)\geq -\eta$.

Given $g\in \widetilde{\VVV} \cap \Diff^{1+}(M)$, let $\beta_1(x)=\max\{ \varphi_j^c(x) \}_{j=1}^{i}$ and $\widehat{\beta}_2(x)=\min\{ \widehat{\varphi}_j^c(x) \}_{j=i+1}^{k}$, where $\varphi_j^c(x)=\log \| Dg |_{E_{j}^c(x)} \|$ and $\widehat{\varphi}_j^c(x)=\log \| Dg^{-1} |_{E_j^c(x)} \|$ for $j=1,\dots,k$. For a ergodic invariant measure $\mu$, we have $\gamma_1(g,\mu)=\int \beta_1(x) \,d\mu$ and $\gamma_2(g^{-1},\mu)=\int \widehat{\beta}_2(x) \,d\mu$. %We write $\beta_2(x)=\max\{\varphi_j^c(x) \}_{j=i+1}^k$.

Now we can define the define a decomposition for $(M,g)$, they are characterized by
\begin{equation}
    \begin{aligned}
        \PPP_g(r) & = \{ (x,n)\in M\times \NN:\, S_n\beta_1(x)\geq -rn\},\\
        \GGG_g(r) & = \{ (x,n)\in M\times \NN:\, S_k\beta_1(x)< -rk \mbox{ and }  S_{k}^{-1}\widehat{\beta}_2(g^{k-1}x)<-rk, \, \forall \, 0\leq k\leq n\},\\
        \SSS_g(r) & = \{ (x,n)\in M\times \NN:\, S_n^{-1} \widehat{\beta}_2(g^{n-1}x)\geq -rn\} \setminus \PPP_g.
        %& = \{(x,n)\in M\times \NN:\, S_n\beta_2(x)\leq rn\} \setminus \PPP_g.
    \end{aligned}
\end{equation}

Consider the diffeomorphism $f:M\rightarrow M$ with dominated splitting $TM=E^s\oplus E_1^c \oplus \cdots \oplus E_k^c \oplus E^u$. We denote $E^{cs}=E^s \oplus E_1^c \oplus\cdots\oplus E_i^c$ and $E^{cu}=E_{i+1}^c \oplus\cdots\oplus E_k^c \oplus E^u$.

\begin{theorem}\label{thm:k-disks}
    There exist a $C^1$ neighborhood $\widetilde{\UUU}$ of $f$, a constants $\rho=\rho(f)>0$, $C=C(f)>0$ and $\eps^{\prime}>0$ such that for every $g\in \widetilde{\UUU}$, we have a continuous collection of trapped plaque families $\widetilde{\WWW}_{g}^{cs}$ tangent to $E_g^{cs}$ with diameter smaller than $\rho$.

    Moreover, if $(x,n)\in \GGG_g(r)$, then the following statements holds:
    \begin{itemize}
        \item[(1)] There are two $C^1$ embedded disks $W_{\rho}^{cs}(g,x)$ and $W_{\rho}^{cu}(g,x)$ which dimensions equal to $\dim E^{cs}$ and $\dim E^{cu}$ respectively, and radius $\rho>0$, centered at $x$, such that
        \begin{equation*}
            T_x W_{\rho}^{cs}(g,x)=E^{cs}(g,x) \mbox{ and } T_x W_{\rho}^{cu}(g,x)=E^{cu}(g,x);
        \end{equation*}
        \item[(2)] $W_{\rho}^{cs}(g,x)$ and $W_{\rho}^{cu}(g,x)$ depend continuously on both $x$ and $g$ in the $C^1$ topology;
        \item[(3)] If $y\in W_{\rho}^{cs}(g,x)$ and $z\in W_{\rho}^{cu}(g,x)$, then
        \begin{equation*}
            d(g^kx,g^ky)\leq C e^{-kr/2} d(x,y) \mbox{ and } d(g^{n-k}x,g^{n-k}z)\leq C e^{-kr/2} d(g^n x,g^n z),
        \end{equation*}
        for every $0\leq k \leq n$;
        \item[(4)] $\widetilde{\WWW}_{g}^{cs}(x)\subset W_{\rho}^{cs}(g,x)$ for every $g\in \widetilde{\UUU}$;
        \item[(5)] If $d(x,y)< \eps^{\prime}$, then
        \begin{equation}\label{eq:k-LPS}
            \widetilde{\WWW}_{g}^{cs}(y) \cap W_{\rho}^{cu}(g,x) \neq \varnothing,
        \end{equation}
        where $y$ is not necessary in some $(\GGG_g(r))_n$. Furthermore, if $z\in \widetilde{\WWW}_{g}^{cs}(y) \cap W_{\rho}^{cu}(g,x)$, then $d(x,z) \leq C d(x,y)$.
    \end{itemize}
\end{theorem}

Let $(x,n)\in \GGG_g(r)$ and $y$ be a contained in a Bowen ball of $x$. We have no idea about information which like item (3) of $y$. To overcome this obstruction, we can use uniform continuity of $\beta_1$ and $\widehat{\beta}_2$. Given $r>0$, let $\eps =\eps(r) >0$ be small enough so that for any $x,y \in M$, we have
\begin{equation}
    d(x,y) < \eps \Rightarrow | \beta_1(x)-\beta_1(y) | \leq \frac{r}{2} \mbox{ and } | \widehat{\beta}_2(x)-\widehat{\beta}_2(y) | \leq \frac{r}{2}.
\end{equation}
In particular, if $x\in \GGG_g(r)$ and $y\in B_n(x,\eps)$, then
\begin{equation*}
    S_k \beta_1(y) \leq S_k \beta_1(x) + rk/2 < -rk/2 , \mbox{ for every } 0\leq k \leq n.
\end{equation*}
and
\begin{equation*}
    S_k^{-1} \widehat{\beta}_2(y) \leq S_k^{-1} \widehat{\beta}_2(x) + rk/2 < -rk/2 , \mbox{ for every } 0\leq k \leq n.
\end{equation*}
So that $y\in \GGG_g(r/2)$. Therefore we have certain arguments in the following.

\begin{proposition}{\rm (Specification)}\label{prop:specification+}
    Let $f$ has minimal stable foliation and $\delta> 0$, there exist a $C^1$ neighborhood $\UUU$ of $f$ such that if $g\in \UUU\cap \Diff^{1+}(M)$ then g has specification on $\GGG_g(r)$ at scale $\delta>0$.
\end{proposition}

\begin{proposition}{\rm (The Bowen proeprty)}\label{prop:bowen+}
    There exist $\eps^{\prime}>0$ and a $C^1$ neighborhood $\UUU$ of $f$ such that if $g \in \UUU \cap \Diff^{1+}(M)$ and potential $\phi: M\rightarrow \RR$ is H\"older continuous, then $\phi$ has the Bowen property on $\GGG_g(r)$.
\end{proposition}

We take $r=\eta$ and $g$ as in \lemref{lem:less-press+}. Let's consider the decomposition $(\PPP_g(\eta), \GGG_g(\eta), \SSS_g(\eta))$.

\begin{proposition}{\rm (Pressure of obstruction to expansivity)}\label{prop:expansivity+}
    There exist $\eps_1 > 0$ and a $C^{1}$ neighborhood $\UUU$ of $f$ such that if $g \in \UUU$, then $\Pexp(g,\phi,\eps_1) < P(g, \phi)$.
\end{proposition}

Let $\widetilde{\VVV}$ be a $C^1$ neighborhood of $f$ as in \lemref{lem:less-press+}.
\begin{proposition}{\rm (Pressure of ``bad'' set)}\label{prop:bad set+}
    There exists $\eps_2 > 0$ such that $P(\PPP_g(\eta) \cup \SSS_g(\eta), g, \phi, \delta, \eps) < P(g, \phi)$ for every $0 < \eps < \eps_1$ and $g \in \widetilde{\VVV}$.
\end{proposition}

\subsection{Proof of \thmref{thm:main+}}
Let $\UUU$ be a $C^1$ neighborhood of $f$ satisfying the above discussion. We choose a $\eps$ such that $0<\eps < \eps^{\prime}, \widehat{\eps}, \eps_1,\eps_2$, where $\eps^{\prime}, \widehat{\eps}, \eps_1$ and $\eps_2$ as in \propref{prop:bowen+}, \propref{prop:expansivity+} and \propref{prop:bad set+}, respectively. For every $g\in \UUU\cap \Diff^{1+}(M)$, there exists a decomposition $(\PPP_g(\eta), \GGG_g(\eta), \SSS_g(\eta))$ for $X\times \NN$ such that $\GGG_g(\eta)$ has the Bowen property at scale $\eps$ and
\begin{equation*}
    \begin{split}
        &\Pexp(g,\phi,\eps) < P(g,\phi),\\
        &P(\PPP_g(\eta)\cup \SSS_g(\eta), g, \phi, \delta, \eps)< P(g,\phi),
    \end{split}
\end{equation*}
for every $g\in \UUU\cap \Diff^{1+}(M)$ and $\delta > 0$. Applying \propref{prop:specification+}, we obtain that $g$ has specification on $\GGG_g(\eta)$ at scale $\delta$ for every $g\in \UUU\cap \Diff^{1+}(M)$ with $\delta< 2000\eps$. Therefore, by \thmref{thm:ct-criteria}, $(g,\phi)$ has a unique equilibrium state.

\subsection*{Acknowledgement}
We are thankful to Dong Chen for helpful discussions on this work.

\bibliographystyle{amsalpha}
\bibliography{ref}

\end{document}